\documentclass[12pt,reqno]{amsart}

\usepackage{blindtext}
\usepackage{geometry}
\usepackage{times}
\usepackage[T1]{fontenc}
\usepackage{mathrsfs}
\usepackage{latexsym}
\usepackage[dvips]{graphics}
\usepackage{epsfig}
\usepackage{blindtext}
\usepackage{microtype}
\usepackage{breqn}
\usepackage[breaklinks]{hyperref}
\usepackage{spverbatim}
\usepackage{xcolor}
\usepackage{physics}
\usepackage{tikz}
\usepackage{listings}
\usepackage{amsmath,amsfonts,amsthm,amssymb,amscd}
\input amssym.def
\input amssym.tex
\usepackage{color}

\newcommand\be{\begin{equation}}
\newcommand\ee{\end{equation}}
\newcommand\bea{\begin{eqnarray}}
\newcommand\eea{\end{eqnarray}}
\newcommand\bi{\begin{itemize}}
\newcommand\ei{\end{itemize}}
\newcommand\ben{\begin{enumerate}}
\newcommand\een{\end{enumerate}}

\newtheorem{thm}{Theorem}[section]

\newtheorem{exa}[thm]{Example}
\newtheorem{defi}[thm]{Definition}

\newtheorem{rek}[thm]{Remark}

\newcommand{\tbf}[1]{\textbf{#1}}

\numberwithin{equation}{section}

\begin{document}

\title{Benfordness of the Generalized Gamma Distribution}

\author{Zelong Bi}
\email{zelong.bi@student.unsw.edu.au} \address{School of Mathematics \& Statistics, University of New South Wales, Sydney NSW, 2052}
\author{Irfan Durmi\'{c}}
\email{Irfan.Durmic@williams.edu} \address{Department of
  Mathematics and Statistics, Williams College, Williamstown, MA 01267}
\author{Steven J. Miller}
\email{Steven.J.Miller@williams.edu} \address{Department of
  Mathematics and Statistics, Williams College, Williamstown, MA 01267}


\thanks{This work was done as part of the Benford Summer 2021 Group led by Professor Steven J. Miller. We would like to thank everyone involved with the Polymath Jr. REU for enabling students to do meaningful research even during a pandemic. The second author was supported by the Williams College John \& Louise Finnerty Class of 1971 Fund for Applied Mathematical Research}

\begin{abstract}
 The generalized gamma distribution shows up in many problems related to engineering, hydrology as well as survival analysis. Earlier work has been done that estimated the deviation of the exponential and the Weibull distribution from Benford's Law. We give a mathematical explanation for the Benfordness of the generalized gamma distribution and present a measure for the deviation of the generalized gamma distribution from the Benford distribution. 
\end{abstract}

\date{\today}

\maketitle

\tableofcontents


\section{Introduction and an Overview of the Theory} \label{sec:overviewoftheory}

At the dawn of the $20^{th}$ century, the astronomer and mathematician Simon Newcomb 
observed that the logarithmic books at his workplace showed a lot of wear and tear at the 
early pages, but the more he progressed through the book, the less usage could be observed.
Newcomb deduced that his colleagues had a "bias" towards numbers starting with the digit 
$1$. In particular, the digit $1$ shows up as the first digit roughly $30\%$ of the time, 
the digit $2$ about $9\%$ of the time, and so on. While he did come up with a mathematical 
model for this interesting relationship, his work stayed mostly unnoticed.

It took another 57 years after Newcomb's discovery for physicist Frank Benford to make the exact same observation as Newcomb: the first pages of logarithmic tables were used far more than others. He formulated this law as follows. 
\begin{defi}\cite[Page 554]{Ben} \label{def:benforiginal}
The frequency of first digits follows closely the logarithmic relation:
\bea
  F_d \ = \log_{10}{\frac{d \ + \ 1}{d}}, \label{eqn:benorg}
\eea
where $d$ represents the leading digit, and $F_d$ represents the frequency of the digit $d$.
\end{defi}

Nowadays, Benford's Law is used in detecting many different forms of fraud, and its prevalence in the world fascinates not only mathematicians, but many other scientists as well (to learn more about Benford's Law and its many applications, we recommend~\cite{BeHi, Nig, Mil} to name a few). 


Many mathematicians have tried to explain the prevalence of Benford's law in the real 
world. Some have shown that samples coming from certain probability distributions tend to 
demonstrate Benfordness~(\cite{MiNi2,CLM12}). We adopt the same methodology in this article
and prove that data coming from a generalized gamma distribution is likely to be close to
Benford's law, and provide an explicit formula to bound the deviation. 

\begin{rek} \label{rek:meaningofbenfordness}
It is worthwhile to make a quick comment about Benfordness. For any finite set, it is impossible to have a perfect fit so we analyze whether or not such a data set is close to Benford which is enough for most applications. If we let the size of the data set tend to infinity, then there is a chance it will converge to Benford. 
\end{rek}


We start with a quick overview of the theory.

\subsection{Benford's Law} \label{subsec:benflaw}


While the formulation of Benford's Law, as presented in Definition~\ref{def:benforiginal}, does have its merits, we want to move away from the idea of frequency and data sets to develop a more probabilistic formualation of the theory. We first present a more complete definition of the Benford distribution for base $B$.
\begin{defi}\label{def:benflaw}
A set of numbers is said to satisfy Benford's law if the leading digit $d \in \{1, 2, ... , B - 1\}$ occurs with frequency $F_d \ = \log_{B}{\frac{d \ + \ 1}{d}}$, where $B \geq 2$.
\end{defi}

\begin{rek}\label{rek:multipledigits}
Note that this is the most common way of stating Benford's law.  A more general version of the law also describes the frequencies of different second digits, third digits and so on. One could also give an expression for the probability of a digit occuring in the $n^{th}$ spot of a number, as is described in the first chapter of~\cite{Mil}. This is related to the so-called Strong Benford's Law which states that the probability of observing a significand of at most $x$ in base $B$ is equal to $\log_{B}{x}$. The distribution of just the first digit, as well as the distribution of the entire significand, is often referred to as just Benford's law. For the purposes of our paper, this difference is inconsequential.
\end{rek}


We now move on to the idea of base and scientific notation. Given a base $B \geq 2$, any nonzero real number $r$ can be uniquely expressed in the form $r \  =  \ aB^n$, where $|a| \in [1, B), n \in \mathbb{Z}$. This is usually referred to as \tbf{scientific notation}, and it motivates the definition of the \tbf{significand}.

\begin{defi} \label{def:significand}
Given a base $B \geq 2$, we define the significand as the mapping $S_B: \mathbb R_{\ne 0} \rightarrow [1,B)$, where $S_B(x)$ is the significand of any input $x \  =  \ aB^n$ written in scientific notation with $|a| \in [1,B)$ and $n \in \mathbb{Z}$. It follows that $S_B(x) \ = \ |a|$ is the \tbf{significand} and $n$ is the \tbf{exponent}.
\end{defi}

One also studies the \tbf{mantissa}, which is the fractional part of the logarithm. 

\begin{exa} \label{exa:digit}
As an example, let $x \ = \ 31295192$. When we write this in scientific notation using base $10$, it becomes $x \ = \  3.1295192 \cdot 10^7$ and it follows that $S_B(3.1295192 \cdot 10^{7}) \ = \ 3.1295192$ is the significand, and $7$ is the exponent. Furthermore, observe that $\log_{10}{31295192} \ \approx \ 7.495477620349604$ so the mantissa is about $0.495477620349604$.
\end{exa}

We now formally introduce the notion of what it means for a random variable to have the Benford distribution \footnote{Note that, throughout this paper, we denote that a random variable $X$ follows the Benford distribution base $B$ by stating that $X$ is Benford base $B$}. 

\begin{defi} \label{def:benford_1st}
A random variable $X$ is Benford base $B$ if $\mathrm{Prob}\left (X \leq x \right) \ = \ \log_{B} x$, where $x \in [1, B)$, $B \geq 2$. 
\end{defi}

Let $X: \Omega \rightarrow \mathbb{R}$, be a random variable with some particular cumulative distribution function (cdf) $F$. If the random variable $S_B \circ X$ is Benford (or close to Benford), then one would expect a data set coming from a population with distribution $F$ to satisfy Benford's law since

\bea \label{eqn:benfdefrandomvariable}
\mathrm{Prob} \left(X \text{ has leading digit } d \right) \ = \ \mathrm{Prob}\left(S_B \circ X \in [d, d + 1)\right) \ \ = \ \ \log_{B}\frac{d + 1}{d},
\eea
which is just a direct application of Definition~\ref{def:benford_1st}.

We could directly find the distribution of $S_B \circ X$ and compare it to the Benford distribution. Alternatively, the following well-known theorem, which can be found in~\cite{Dia}, provides an indirect method, which in some cases is more convenient.

\begin{thm} \label{thm:logmod1}
Given a base $B \geq 2$ and a non-negative random variable $X$, $S_B \circ X$ is Benford if and only if $\log_B X \bmod 1$ has a uniform $[0,1)$ distribution.
\end{thm}

\begin{proof} \label{prf:logmod1}
For any $s \in [1, B)$, let $u \  =  \ \log_B s \in [0,1)$. Let us assume that $S_B \circ X$ is Benford, so the following holds:
\bea \label{benfderv1} 
\mathrm{Prob}\left(\log_B X \bmod 1 \in [0,u)\right)  & \ = \ & \mathrm{Prob}\left(\{X \in [1 \cdot B^k, s \cdot B^k): k\in 
\mathbb{Z}\}\right)\nonumber\\ & \ = \ & \mathrm{Prob}\left(S_B \circ X \in [1, s)\right) \nonumber \\ & \ = \ &
 \log_B s \nonumber \\ & \ = \ & u.
\eea
Hence it follows that $\log_B X \bmod 1$ has uniform $[0,1)$ distribution.


\noindent Let us now assume that  $\log_B X \bmod 1$ has uniform $[0,1)$ distribution. Then:
\bea \label{benfderv2} 
\mathrm{Prob}\left(S_B \circ X \in [1, s)\right)  & \ = \ & \mathrm{Prob}\left(\{X \in [1 \cdot B^k, s \cdot B^k): k\in 
\mathbb{Z}\}\right)\nonumber\\ & \ = \ & \mathrm{Prob}\left(\log_B X \bmod 1 \in [0,u)\right) \nonumber \\ & \ = \ &
u  \nonumber \\ & \ = \ & \log_B s.
\eea
\end{proof}

Theorem~\ref{thm:logmod1} forms the key foundation of our work since it reduces our problem substantially by allowing us to focus on a logarithmically rescaled random variable  $\bmod 1$. We use this result  to explore the relationship between the generalized gamma distribution and Benford's law.

It is exactly Theorem~\ref{thm:logmod1} which enables us to develop a measure for the 
deviation of a random variable from the Benford distribution. This deviation, which we 
denote as $D$, quantifies the deviation of the leading digits from their corresponding 
Benford counterparts. We derive an estimate for this deviation. 

Assume $X$ is a nonnegative random variable and that $S_B \circ Y$ is Benford. We now transform the random variable $X$ via the transformation noted in Theorem~\ref{thm:logmod1} to $\log_B X \bmod 1$, and denote its probability density function (pdf) as: $f_{\log_B X \bmod 1}$. From here it follows that 
\bea \label{benfdiff} 
D & \ = \ & \left|\mathrm{Prob}\left(X \text{ has leading digit } d \right) - \log_B \frac{d + 1}{d} \right|\nonumber \\
& \ = \ & |\mathrm{Prob}\left(S_B \circ X \in [d, d + 1)\right)  -  \mathrm{Prob}\left(S_B \circ Y \in [d, d + 1)\right)| \nonumber\\
& \ = \ & \left|\mathrm{Prob}\left(\{X \in [d \cdot B^k, (d + 1) \cdot B^k): k\in \mathbb{Z}\}\right) - \mathrm{Prob}\left(\{Y \in [d \cdot B^k, (d + 1) \cdot B^k): k\in \mathbb{Z}\}\right)\right|\nonumber\\
& \ = \ & \left|\mathrm{Prob}\left(\log_B X \bmod 1 \in [\log_B d,\log_B d + 1)\right) - \mathrm{Prob}\left(\log_B Y \bmod 1 \in [\log_B d,\log_B d + 1)\right)\right| \nonumber\\
& \ = \ &  
\left|\int_{\log_B d}^{\log_B d + 1} f_{\log_B X \bmod 1}(u) - 1 \ du\right|\nonumber\\
& \ \leq \ & \int_{0}^{1} \left|f_{\log_B X \bmod 1}(u) - 1 \right|du.
\eea

This expression allows us to bound the deviation of any random variable from the Benford distribution. We now present a quick theoretical overview of the generalized gamma distribution.


\subsection{The Generalized Gamma Distribution and Its Connection to Benford} \label{subsec:generalized gamma}

The work done by Miller and Nigrini in~\cite{MiNi2}, as well as the paper by Leemis, Schmeiser, and Evans~\cite{LSE}, explored the exponential distribution and how it relates to Benford's Law, whereas Cuff et.al in~\cite{CLM12} explored a similar form of a relationship between the Weibull distribution and Benford's Law. Both of these distributions can be seen as ``children'' of one parent distribution for a particular choice of parameters. 

For the purposes of this paper, we use the following definition of the generalized gamma distribution, as presented in~\cite{Sta}.

\begin{defi} \label{def:gengamma}
A random variable $X$ follows the generalized gamma distribution with parameters $a$, $d$, and $p$ if its cumulative distribution function (cdf) is of the form
\bea \label{eqn:gammacdf}
F(x;a,d,p) \ = \ \frac{\gamma \left(\frac{d}{p}, \left(\frac{x}{a}\right)^p\right)}{\Gamma\left(\frac{d}{p}\right)}, \ \ x > 0; \  a,d,p > 0,
\eea
where $\gamma$ is the lower incomplete gamma function, defined as
\bea \label{eqn:lowerincompgamma}
\gamma(s, x) \ := \ \int_0^x t^{s-1}e^{-t}dt.
\eea
The corresponding probability density function (pdf) is
\bea \label{eqn:gammapdf}
f(x;a,d,p) \ = \ \frac{\left(\frac{p}{a^d}\right)x^{d-1}e^{-(x/a)^p}}{\Gamma \left(\frac{d}{p}\right)}.
\eea
\end{defi}

We note that, when $d \ = \ p$, Equation~\eqref{eqn:gammacdf} is just the cdf of a Weibull distribution, and that is further reduced to the exponential distribution for the special case of $d \ = \ p \ = \ 1$. This is a very useful observation because it enables us to directly compare our results with the work completed in~\cite{MiNi2,CLM12}. 

Having covered the relevant background material, our goal now is to show the following key results, that we prove in Section~\ref{sec:mainresults}.

\begin{thm} \label{thm:main}
If $X$ is a random variable having the generalized gamma distribution with parameters $a, d, p$, then the pdf of $\log_B X \bmod 1$ is
\bea \label{logpdf1}
f_{\log_B X \bmod 1}(u) \ = \ \frac{p\ln B}{\Gamma\left(\frac{d}{p}\right)}\sum_{k=-\infty}^{+\infty}e^{-\left(\frac{B^{k+u}}{a}\right)^p}\left(\frac{B^{k+u}}{a}\right)^d,
\eea
or
\bea \label{logpdf2}
f_{\log_B X \bmod 1}(u) \ = \ 1 + \sum_{k=1}^{+\infty}\frac{2}{\Gamma\left(\frac{d}{p}\right)}\Re\left[e^{2\pi ui - \frac{2\pi i\ln a }{\ln B}}\Gamma\left(\frac{d}{p} - \frac{2\pi k i}{p\ln B}\right)\right].
\eea
where $u \in (0,1)$. Further, the scaling parameter $a$ has limited effect on the pdf, for any $m \in \mathbb{Z}$, $a$ and $a \cdot B^m$ result in the same pdf.
\end{thm}

\begin{thm} \label{thm:error}
Given $\epsilon > 0$, and the second form of the pdf of $\log_B X \bmod 1$, i.e. Equation~\eqref{logpdf2} in Theorem~\ref{thm:main}, we have
$$f_{\log_B X \bmod 1}(u) \ \ = \ \ 1 + \sum_{k=1}^{+\infty}\frac{2}{\Gamma\left(\frac{d}{p}\right)}\Re\left[e^{2\pi ui - \frac{2\pi i\ln a }{\ln B}}\Gamma\left(\frac{d}{p} - \frac{2\pi k i}{p\ln B}\right)\right].$$
To approximate this function with main term $1$ and first $M$-term partial sum of the residue
\bea \label{psum}
f_{\log_B X \bmod 1}^{M}(u) & \ = \ & 1 + \sum_{k=1}^{M}\frac{2}{\Gamma\left(\frac{d}{p}\right)}\Re\left[e^{2\pi ui - \frac{2\pi i\ln a }{\ln B}}\Gamma\left(\frac{d}{p} - \frac{2\pi k i}{p\ln B}\right)\right],
\eea
to make sure the approximation error at any point $u \in (0,1)$ is bounded by $\epsilon$, $M$ should satisfy
\bea \label{Mbound}
M & \ > \ & \frac{(d + p)^2(\ln(B))^2}{2 \pi^2 \epsilon} - 1.
\eea
\end{thm}

One can gain a lot of intuition for the behavior of the distribution by analysing the graph of the pdf for different parameters. In particular, the parameters $d$ and $p$ determine the shape of the pdf, while the parameter $a$ determines the spread of the pdf. 

We also present some simulations, along with figures, that show how close the generalized gamma distribution comes to Benford's Law. The purpose of these figures is to show us that the deviation from Benford should be relatively low. 

Figure~\ref{fig:gammasample} compares the first-digit frequencies of $10000$ samples from a generalized gamma distribution, using the parameters: $B \ = \ 10, a \ = \ 2, d \ = \ 1, p \ = \ \frac{1}{2} $, with the frequencies predicted by Benford's law. Observe that Benford's law provides us with a remarkably good fit.

\begin{figure}[h]
\includegraphics[scale=0.55]{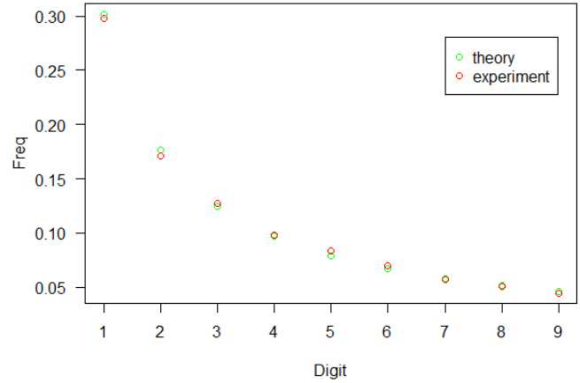}
\caption{First-digit frequencies of 10000 samples from a generalized gamma distribution $\left(B \ = \ 10, a \ = \ 2, d \ = \ 1, p \ = \ \frac{1}{2} \right).$}
\label{fig:gammasample}
\end{figure}


The Benfordness of samples coming from the generalized gamma distribtion can also be observed in a different way. We already stated in Theorem~\ref{thm:logmod1} that a random 
variable $X$ is Benford if and only if $\log_B X \bmod 1$ is uniform $[0,1)$. Since our claim is that $X$ is close to Benford if it follows the generalized gamma distribution, then
$\log_B X \bmod 1$ should have to be close to the uniform $[0,1)$ distribution. We explore this further using a Kolmogorov-Smirnov test.

The Kolmogorov-Smirnov test is used to examine whether or not a sample comes from a population with a specific distribution. The smaller the test statistic is, the more likely the sample came from the target distribution. We generated samples from the generalized gamma distribution with different values of parameters $d$ and $p$, and performed a Kolmogorov-Smirnov test to compare the transformed data ($\log_B X \bmod 1$) with the uniform $[0,1)$ distribution. The result is shown in Figure~\ref{fig:KS-test}. Observe that the test statistic is pretty small, indicating the (transformed) data came from a population with an approximately uniform $[0,1)$ distribution, and hence the original distribution is close to Benford. The match is better when $d$ and $p$ are small, which is reasonable considering that we know how close the exponential and Weibull distributions come to the Benford distribution. 

\begin{figure}[h]
\includegraphics[scale=.55]{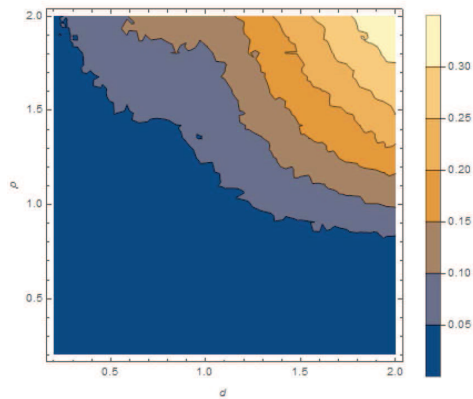}
\caption{Kolmogorov-Smirnov test results under different values of $d$ and $p$.}
\label{fig:KS-test}
\end{figure}



\section{Main Results and Key Observations} \label{sec:mainresults}

In this section, we prove the results of Theorem~\ref{thm:main} and Theorem~\ref{thm:error}, and justify the observations shown in Figures~\ref{fig:gammasample} and~\ref{fig:KS-test}.


\begin{proof}[Proof of Theorem \ref{thm:main}]
Given $u \in [0,1)$, we have
\bea \label{logpdfderv1} 
\mathrm{Prob}(\log_B X \bmod 1 \in [0,u]) & \ = \ & \sum_{k=-\infty}^{\infty}\mathrm{Prob}(\log_B X \in [k, k+u])\nonumber \\
& \ = \ & \sum_{k=-\infty}^{\infty}\mathrm{Prob}(X \in [B^k, B^{k+u}])\nonumber \\ & \ = \ & \frac{1}{\Gamma\left(\frac{d}{p}\right)}\sum_{k=-\infty}^{\infty}\int_{\left(\frac{B^k}{a}\right)^p}^{\left(\frac{B^{k+u}}{a}\right)^p}t^{\frac{d}{p} - 1}e^{-t} dt.
\eea

By using some results from analysis, the following can be verified.
\begin{enumerate}
\item The resulting function~\eqref{logpdfderv1} converges for all $u \in [0,1)$.
\item The resulting function~\eqref{logpdfderv1} is differentiable for all $u \in (0,1)$, and we can differentiate it term by term.
\end{enumerate}
Appendix~\ref{appendix:details_of_proof} provides the interested reader with more details of the proof as well as the mathematical machinery used in the paper.

We now work with the integral from~\eqref{logpdfderv1} to get the probability density function. Using the Fundamental Theorem of Calculus we find
\bea \label{logpdfderv2} 
\dv{u} \int_{\left(\frac{B^k}{a}\right)^p}^{\left(\frac{B^{k+u}}{a}\right)^p}t^{\frac{d}{p} - 1}e^{-t} dt & \ = \ & e^{-\left(\frac{B^{k+u}}{a}\right)^p}\left(\left(\frac{B^{k+u}}{a}\right)^p\right)^{\frac{d}{p} - 1}p\left(\frac{B^{k+u}}{a}\right)^{p-1}\frac{B^{k+u}\ln{B}}{a}\nonumber \\
& \ = \ & e^{-\left(\frac{B^{k+u}}{a}\right)^p}\left(\frac{B^{k+u}}{a}\right)^d p \ln{B},
\eea
and then plugging~\eqref{logpdfderv2} back into~\eqref{logpdfderv1}, we get that the pdf of $\log_B X \bmod 1$ is
\bea \label{logpdfresult} 
f_{\log_B X \bmod 1}(u) \ = \ \frac{p\ln B}{\Gamma\left(\frac{d}{p}\right)}\sum_{k=-\infty}^{+\infty}e^{-\left(\frac{B^{k+u}}{a}\right)^p}\left(\frac{B^{k+u}}{a}\right)^d.
\eea

Next we apply Poisson summation to~\eqref{logpdf1} to get the equivalent form~\eqref{logpdf2}, which is better since it is divided into a main term, $1$, which is what we want, and a residue term given by an infinite series. See Appendix~\ref{appendix:details_of_proof} or~\cite{CLM12} for details about Poisson summation.

For any $u \in (0,1)$, let $z \ = \ B^u$, $t \ = \ k$. 
We claim that 
\bea \label{claimforg(t)}
g(t) \ = \ p\ln{B}e^{-\left(\frac{B^tz}{a}\right)^p}\left(\frac{B^tz}{a}\right)
\eea
satisfies the conditions for applying Poisson summation. The details of why this is true can be found in Appendix~\ref{appendix:details_of_proof}, but the important insight is that we have the Fourier transform of this function:
\bea \label{logpdfderv3} 
\hat{g}(f) & \ = \ & \int_{-\infty}^{+\infty}p\ln{B}e^{-\left(\frac{B^tz}{a}\right)^p}\left(\frac{B^tz}{a}\right)^de^{-2\pi i tf} dt\nonumber \\
& \ = \ & \int_{0}^{\infty}e^{-\omega}\omega^{\frac{d}{p} - 1}\left(\frac{a\omega^{\frac{1}{p}}}{z}\right)^{-\frac{2\pi if}{\ln{B}}} d\omega \ \ \ ,\text{ where $\omega \ = \ \left(\frac{B^tz}{a}\right)^p$} \nonumber \\ & \ = \ & \left(\frac{z}{a}\right)^{\frac{2\pi i f}{\ln{B}}}\Gamma\left(\frac{d}{p} - \frac{2\pi i f}{p\ln{B}}\right).
\eea
We now proceed to apply Poisson summation to Equation~\eqref{logpdf1}, and at the end we get Equation~\eqref{logpdf2}

\bea \label{logpdfderv4}
f_{\log_B X \bmod 1}(u) & \ = \ & \frac{1}{\Gamma\left(\frac{d}{p}\right)}\sum_{k=-\infty}^{\infty}g(k)\nonumber\\
& \ = \ & \frac{1}{\Gamma\left(\frac{d}{p}\right)}\sum_{k=-\infty}^{\infty}\hat{g}(k)\nonumber\\
& \ = \ & \frac{1}{\Gamma\left(\frac{d}{p}\right)}\sum_{k=-\infty}^{\infty}\left(\frac{z}{a}\right)^{\frac{2\pi i k}{\ln{B}}}\Gamma\left(\frac{d}{p} - \frac{2\pi i k}{p\ln{B}}\right) \nonumber\\
& \ = \ & 1 + \sum_{k=1}^{\infty}\left[\left(\frac{z}{a}\right)^{\frac{2\pi i k}{\ln{B}}}\Gamma\left(\frac{d}{p} - \frac{2\pi i k}{p\ln{B}}\right) + \left(\frac{z}{a}\right)^{-\frac{2\pi i k}{\ln{B}}}\Gamma\left(\frac{d}{p} + \frac{2\pi i k}{p\ln{B}}\right)\right]\nonumber\\
& \ = \ & 1 + \sum_{k=1}^{+\infty}\frac{2}{\Gamma\left(\frac{d}{p}\right)}\Re\left[\left(\frac{z}{a}\right)^{-\frac{2\pi i k}{\ln{B}}}\Gamma\left(\frac{d}{p} - \frac{2\pi k i}{p\ln B}\right)\right] \ \ \text{(Note that $\Gamma\left(\overline{z}\right) \ = \ \overline{\Gamma\left(z\right)}$)} \nonumber\\
& \ = \ & 1 + \sum_{k=1}^{+\infty}\frac{2}{\Gamma\left(\frac{d}{p}\right)}\Re\left[e^{2\pi ui - \frac{2\pi i\ln a }{\ln B}}\Gamma\left(\frac{d}{p} - \frac{2\pi k i}{p\ln B}\right)\right].
\eea

Finally, observe that $e^{2\pi i x} \ = \ e^{2 \pi i (x + 1)}$ for all $x \in \mathbb{R}$, which verifies the scaling invariant property of $a$.
\end{proof}

\begin{rek} \label{rek:main}
It is worth noting that, when we reduce the results from Equation~\eqref{logpdf1} and Equation~\eqref{logpdf2} to the Weibull case, we retrieve the same results that were shown in the article by Cuff et.al~\cite{CLM12}.
\end{rek}

The following result, expressed as Theorem~\ref{thm:error}, enables us to estimate the value of the pdf numerically.


\begin{proof} [Proof of Theorem~\ref{thm:error}]
For any $M \geq 1$, the approximation error is
\bea \label{error1}
|\eta| & \ = \ & \left|\sum_{k=M+1}^{+\infty}\frac{2}{\Gamma\left(\frac{d}{p}\right)}\Re\left[e^{2\pi ui - \frac{2\pi i\ln a }{\ln B}}\Gamma\left(\frac{d}{p} - \frac{2\pi k i}{p\ln B}\right)\right]\right|\nonumber\\
& \ \leq \ & \sum_{k=M+1}^{+\infty} \frac{2}{\Gamma\left(\frac{d}{p}\right)}\left|\Gamma\left(\frac{d}{p} - \frac{2\pi k i}{p\ln B}\right)\right|,
\eea
where $u \in (0,1)$. 

The gamma function has the property that
\bea \label{error2}
\left|\Gamma(a + bi)\right|^2 & \ = \ & \left|\Gamma(a)\right|^2\prod_{k=0}^{\infty}\frac{1}{1 + \frac{b^2}{(a + k)^2}},
\eea
and applying this to~\eqref{error1}, we get
\bea \label{error3}
\left|\Gamma\left(\frac{d}{p} - \frac{2\pi k i}{p\ln B}\right))\right|^2 & \ = \ &
\left[\Gamma\left(\frac{d}{p}\right)\right]^2\prod_{l=0}^{\infty}\frac{1}{1 + \frac{\left(\frac{2\pi k}{p\ln{B}}\right)^2}{\left(\frac{d}{p} + l\right)^2}} \nonumber\\
& \ \leq \ & \left[\Gamma\left(\frac{d}{p}\right)\right]^2\prod_{l=0}^{1}\frac{1}{1 + \frac{\left(\frac{2\pi k}{p\ln{B}}\right)^2}{\left(\frac{d}{p} + l\right)^2}}\nonumber\\
& \ \leq \ & \left[\Gamma\left(\frac{d}{p}\right)\right]^2\frac{[(d+p)\ln{B}]^4}{(2\pi k)^4},
\eea
where the first inequality is because all terms in the product are positive numbers less than or equal to $1$. Finally we have
\bea \label{error4}
|\eta| & \ \leq \ & \sum_{k=M+1}^{+\infty} \frac{[(d + p)\ln{B}]^2}{2\pi^2k^2}\nonumber\\
& \ \leq \ & \int_{M+1}^{\infty} \frac{[(d + p)\ln{B}]^2}{2\pi^2x^2} dx\nonumber\\
& \ = \ & \frac{[(d + p)\ln{B}]^2}{2\pi^2(M+1)}.
\eea

Letting $|\eta| < \epsilon$, and using the result from~\eqref{error4}, we get the lower bound for $M$ as in~\eqref{Mbound}.
\end{proof}

With the help of~\eqref{benfdiff} and Theorem~\ref{thm:error}, for a random variable following the generalized gamma distribution with parameters $(a, d, p)$, we have the following for its deviation:
\bea \label{estimatedev}
D & \ = \ & \left|P(X \text{ has leading digit } d ) - \log_B \frac{d + 1}{d} \right|\nonumber\\
 & \ \leq \ & \int_{0}^{1} |f(u) - 1 |du\nonumber\\
& \ \leq \ & \int_{0}^{1} |f(u) - f_M(u) | du + \int_{0}^{1} |f_M(u) - 1 | du\nonumber\\
& \ \leq \ & \epsilon + \sup_{u \in (0,1)}|f_M(u) - 1 |,
\eea
where $f$ and $f_M$ are exact and approximate pdfs of $\log_B X \bmod 1$. Since we can control $\epsilon$ (which then determines $M$), and $\sup_{u \in (0,1)}|f_M(u) - 1 |$ can be evaluated (at least) numerically, we can get an upper bound for the difference of $S_B \circ X$ from the Benford distribution for any given parameters $a$, $d$ and $p$.

Figure~\ref{fig:logpdfs} shows the graphs of some approximate pdfs (with approximation error $< 0.01$) of $\log_B X \bmod 1$ with different parameters, which are pretty close to the constant function $1$. The term $1$ in~\eqref{logpdf2} plays a major role in the function. 

\begin{figure}[h]
\includegraphics[scale=1]{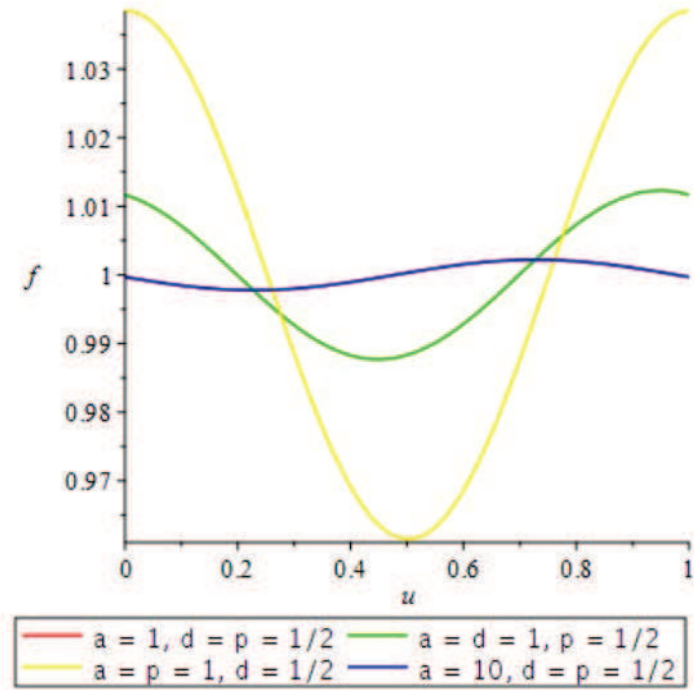}
\caption{Approximation of pdfs of $\log_B X \bmod 1$ with different parameters.}
\label{fig:logpdfs}
\end{figure}

Figures~\ref{fig:estimatedev_d} and~\ref{fig:estimatedev_p} show the upper bound of $D$ with respect to $d$ and $p$ according to~\eqref{estimatedev}. We see $X$ is very close to Benford when the parameters $d$ and $p$ are small, which is consistent with the Kolmogorov-Smirnov test we showed in Figure~\ref{fig:KS-test}.

\begin{figure}[h]
\includegraphics[scale=1]{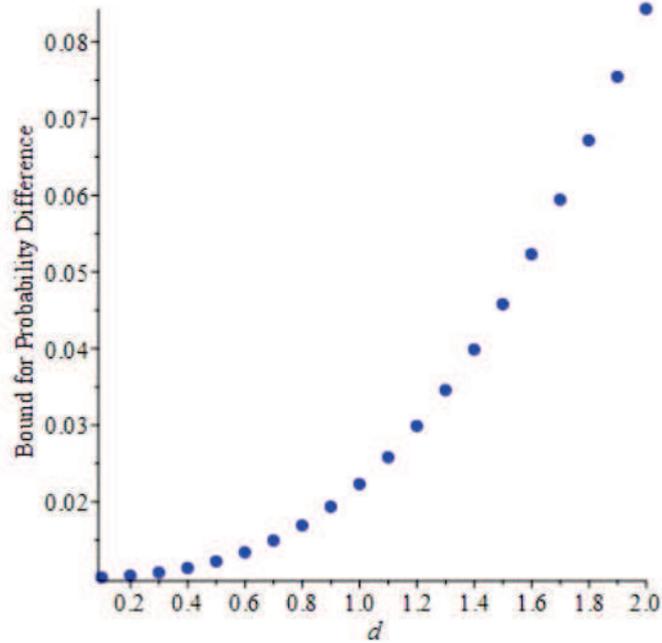}
\caption{Bound of probability difference \eqref{estimatedev} with respect to $d$ \ ($a = 1$, $p = 0.5$).}
\label{fig:estimatedev_d}
\end{figure}

\begin{figure}[h]
\includegraphics[scale=1]{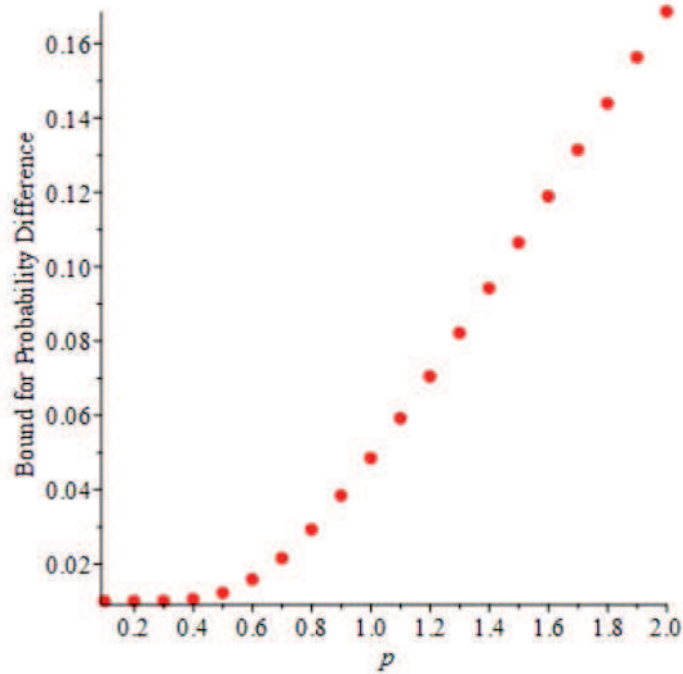}
\caption{Bound of probability difference \eqref{estimatedev} with respect to $p$ \ ($a = 1$, $d = 0.5$).}
\label{fig:estimatedev_p}
\end{figure}


\section{Conclusion and Future Work} \label{sec:conclusionandfuturework}

We have shown that the generalized gamma distribution for the right choice of parameters, meaning a relatively small $p$ and $d$, conforms well to Benford's Law for the leading digit. It would be interesting to see where we can find an application of this result, considering how common the generalized gamma distribution is in nature. 

A possible research avenue would be to perform a similar analysis for other families of distributions. 




\appendix

\section{Details of Proof and Mathematical Machinery Used} \label{appendix:details_of_proof}

In the proof of Theorem~\ref{thm:main}, we claimed the following.
\begin{enumerate}
    \item The cdf of $\log_B X \bmod 1$ \eqref{logpdfderv1} converges pointwise for all $u \in [0,1)$, it is differentiable for all $u \in (0,1)$, and we can differentiate it term by term.
    \item Poisson summation can be applied to the pdf of $\log_B X \bmod 1$ as stated in~\eqref{logpdfresult}.
\end{enumerate}

We use the following results to prove these facts. Many of these proofs are standard and the arguments below are provided in order to give the reader a quick overview of the key ideas. 


\begin{thm}(\textbf{Weierstrass M-test}) \label{MTest}
Let $\{f_n\}_{n=1}^{\infty}$ be a sequence of real valued functions on a set $X$, suppose each $|f_n|$ is bounded by $M_n \geq 0$, then if $\sum_{n=1}^{\infty} M_n$ converges, $\sum_{n=1}^{\infty}f_n$ converges uniformly.
\end{thm}

\begin{thm} \label{termdiff}
If $\{f_n\}_{n=1}^{\infty}$ is a sequence of $C^1$ functions on $(a,b)$, if both $\sum_{n=1}^{\infty}f_n$ and $\sum_{n=1}^{\infty}f'_n$ converge uniformly, then $\sum_{n=1}^{\infty}f_n$ is differentiable and $(\sum_{n=1}^{\infty}f_n)' \ = \ \sum_{n=1}^{\infty}f'_n$.
\end{thm}

\begin{thm}(\textbf{Poisson Summation}) \label{poissum}
Let $f$, $f'$ and $f''$ be continuous functions which eventually decay at least as fast as $x^{-(1+\eta)}$ for some $\eta > 0$, then
\bea \label{eqn:poissum}
\sum_{n=-\infty}^{+\infty} f(n) & \ = \ & \sum_{n=-\infty}^{+\infty} \hat{f}(n),
\eea
where $\hat{f}(y) \ = \ \int_{-\infty}^{+\infty}f(x)e^{-2\pi xyi}dx$ is the Fourier transformation of $f$.
\end{thm}

Our second claim above is a direct result of Theorem~\ref{poissum}. For the first claim to be true, we need to check~\eqref{logpdfderv1} satisfies the conditions in Theorem~\ref{termdiff}.   
\begin{enumerate}
    \item The term in the sum of \eqref{logpdfderv1},
    \bea \label{termincdf}
    \int_{\left(\frac{B^k}{a}\right)^p}^{\left(\frac{B^{k+u}}{a}\right)^p}t^{\frac{d}{p} - 1}e^{-t} dt
    \eea
    is $C^1$ in $(0,1)$.
    \item $F_{\log_B X \bmod 1}$ in \eqref{logpdfderv1} converges uniformly in $(0,1)$.
    \item $f_{\log_B X \bmod 1}$ in \eqref{logpdfresult} converges uniformly in $(0,1)$.
\end{enumerate}
1) is easy to check. The main reason for the uniform convergence of $F_{\log_B X \bmod 1}$ 
and $f_{\log_B X \bmod 1}$ is the fast-decay $e^{-u}$ like term in both of them. Next we 
give a sketch for part of the proof for this. We will show that if $d \geq p$, $F_{\log_B X \bmod 1}$ converges uniformly in $(0,1)$. All other cases could be checked similarly.

\begin{proof} (\textit{Sketch}) For any $u \in (0,1)$, if $d \geq p$, it's easy to check 
$$t^{\frac{d}{p} - 1}e^{-t} \ \leq \  \left(\frac{B^{k+1}}{a}\right)^{d - p}e^{-\left(\frac{B^k}{a}\right)^p},$$
for all $t \in \left[\left(\frac{B^k}{a}\right)^p, \left(\frac{B^{k+u}}{a}\right)^p\right]$, so we have  
\bea
\int_{\left(\frac{B^k}{a}\right)^p}^{\left(\frac{B^{k+u}}{a}\right)^p}t^{\frac{d}{p} - 1}e^{-t} dt  \ \leq \  \left(\frac{B^{k+1}}{a}\right)^{d - p}e^{-\left(\frac{B^k}{a}\right)^p}\left(\frac{B^{k+1}}{a}\right)^{d} & \ \leq \ & \left(\frac{B^{k+1}}{a}\right)^{d}e^{-\left(\frac{B^k}{a}\right)^p}.\nonumber
\eea
Apply this result to the sum in \eqref{logpdfderv1}, we get 
\bea \label{cdfconverge}
\Gamma\left(\frac{d}{p}\right)F_{\log_B X \bmod 1} & \ = \ &
\sum_{k=-\infty}^{\infty}\int_{\left(\frac{B^k}{a}\right)^p}^{\left(\frac{B^{k+u}}{a}\right)^p}t^{\frac{d}{p} - 1}e^{-t} dt\nonumber\\ 
& \ \leq \ & \sum_{k=-\infty}^{\infty} \left(\frac{B^{k+1}}{a}\right)^{d}e^{-\left(\frac{B^k}{a}\right)^p}\nonumber \\
& \ \leq \ & \sum_{k=0}^{\infty} \left(\frac{B^{k+1}}{a}\right)^{d}e^{-\left(\frac{B^k}{a}\right)^p} + \sum_{k=1}^{\infty}\left(\frac{B}{aB^k}\right)^{d}e^{-\left(\frac{1}{aB^k}\right)^p}. 
\eea

The two sums in \eqref{cdfconverge} are convergent by the integration test, then by Weierstrass M-test, i.e. Theorem~\ref{MTest}, we know the original sum in~\eqref{logpdfderv1} converges uniformly.
\end{proof}




\section{Simulation Code} \label{sec:codeused}

\noindent\textbf{R script: Sample from a Generalized Gamma Distribution and compare the first-digit frequencies of the data with values predicted by Benford's law}\\
\begin{lstlisting}[language=R]
N=10000
a = 2
d = 1/2
p = 1/2
B <- 10 # B should be an integer greater than 1

sample <- as.vector(qgamma(runif(N), shape=d/p, scale=a^p)^(1/p))

for (i in 1:N) {
  while (sample[i] < 1 | sample[i] >= B) {
    if (sample[i] < 1) {
      sample[i] <- sample[i] * B
    } else {
      sample[i] <- sample[i] / B
    }
  }
  sample[i] <- trunc(sample[i])
}

freqs <- as.numeric(table(sample))

error <- 0.0
freqs_t <- vector("list", B - 1)
freqs_t <- unlist(freqs_t)

for (i in 1:(B-1)) {
  freqs_t[i] <- logb((i + 1) / i, base=B)
  error <- error + (freqs[i] / N - freqs_t[i])^2
}

freqs <- freqs / N
d <- 1:(B-1)

plot(d, freqs, las=1, xlab="Digit", ylab="Freq", col="red", xaxt="n")
axis(1, at=1:(B-1), labels=1:(B-1))
points(d, freqs_t, col="green")
legend(B - 3, max(c(freqs, freqs_t)) * 0.95, 
legend=c("theory", "experiment"), 
col=c("green", "red"), pch=c(21, 21))
\end{lstlisting}
~\\
\noindent\textbf{Maple code: Plot pdfs and calculate probability deviation bound}
\begin{lstlisting}
restart;
with(plots):

g_k:=2/GAMMA(d/p)*GAMMA(d/p-2*Pi*I*k/p/ln(B))
*exp(2*Pi*I*u- 2*Pi*I*ln(a)/ln(B));
f_M := 1 + sum(Re(g_k), k=1..M);
M := ceil(((d + p)*ln(B))^2/2/Pi/Pi/e - 1);

# set pointwise error bound
e := 0.01; 

# set parameters and plot
B:=10; a:=1; d:=1/2; p:=1/2;
p1 := plot(f_M, u=0..1, color="red", 
labels=[u, f], legend="a = 1, d = p = 1/2"):

B:=10; a:=1; d:=1; p:=1/2;
p2 := plot(f_M, u=0..1, color="green", 
labels=[u, f], legend="a = d = 1, p = 1/2"):

B:=10; a:=1; d:=1/2; p:=1;
p3 := plot(f_M, u=0..1, color="yellow", 
labels=[u, f], legend="a = p = 1, d = 1/2"):

B:=10; a:=10; d:=1/2; p:=1/2;
p4 := plot(f_M, u=0..1, color="blue", 
labels=[u, f], legend="a = 10, d = p = 1/2"):

display(p1, p2, p3, p4, legendstyle = [font=["HELVETICA", 12]
, location=bottom]);

# calculate bounds for probability difference under different parameters
# d = p = 0.5, a changes from 1 to 10
# unassign('a'); d := 0.5; p := 0.5;
# points:={seq([a, Optimization[Maximize]
(abs(f_M - 1),u = 0..1)[1] + e], a=1..10)};
# pointplot(points, symbol=solidcircle, symbolsize = 15, 
color =orange, labels=["a", "Bound for Probability Difference"], 
labeldirections=[ "horizontal", "vertical"]);

# a = 1, p = 0.5, d changes from 0.1 to 2
unassign('d'); a := 1; p := 0.5;
ds := seq(n/10, n=1..20);
points:={seq([d, Optimization[Maximize]
(abs(f_M - 1),u = 0..1)[1] + e], d=ds)};
pointplot(points, symbol=solidcircle, symbolsize = 15, 
color =blue, labels=[d, "Bound for Probability Difference"], 
label directions=[ "horizontal", "vertical"]);


# a = 1, d = 0.5, p changes from 0.1 to 2
unassign('p'); a := 1; d := 0.5;
ps := seq(n/10, n=1..20);
points:={seq([p, Optimization[Maximize]
(abs(f_M - 1),u = 0..1)[1] + e], p=ps)};
pointplot(points, symbol=solidcircle, symbolsize = 15, 
color =red, labels=[p, "Bound for Probability Difference"],
label directions=[ "horizontal", "vertical"]);
\end{lstlisting}
~\\
\noindent\textbf{Mathematica code for the Kolmogorov-Smirnov test}
\begin{lstlisting}[language = Mathematica]
Clear[Diff]
Diff[a_, d_, p_, B_] := KolmogorovSmirnovTest[Mod[Log[B, 
Random Variate[GammaDistribution[d, a, p, 0], 10^4]], 1], 
Uniform Distribution[], "TestStatistic"]
ContourPlot[Diff[1, d, p, 10], {d, 0.2, 2}, {p, 0.2, 2}, 
Frame Label -> Automatic, PlotLegends -> Automatic]
\end{lstlisting}

\end{document}